\documentclass[11pt]{amsart}
\usepackage[all]{xy}
\usepackage{amscd,amsthm,amssymb,amsfonts,amsmath,euscript}

\theoremstyle{plain}
\newtheorem{thm}{Theorem}[section]
\newtheorem{lemma}[thm]{Lemma}
\newtheorem{prop}[thm]{Proposition}

\theoremstyle{definition}
\newtheorem{defn}[thm]{Definition}

\theoremstyle{remark}
\newtheorem{remark}[thm]{Remark}

\newcommand{\nc}{\newcommand}

\def\makeop#1{\expandafter\def\csname#1\endcsname
  {\mathop{\rm #1}\nolimits}\ignorespaces}
\makeop{Hom}   \makeop{End}   \makeop{Aut}   \makeop{Isom}  \makeop{Pic} 
\makeop{Gal}   \makeop{ord}   \makeop{Char}  \makeop{Div}   \makeop{Lie} 
\makeop{PGL}   \makeop{Corr}  \makeop{PSL}   \makeop{sgn}   \makeop{Spf}
\makeop{Spec}  \makeop{Tr}    \makeop{Nr}    \makeop{Fr}    \makeop{disc}
\makeop{Proj}  \makeop{supp}  \makeop{ker}   \makeop{im}    \makeop{dom}
\makeop{coker} \makeop{Stab}  \makeop{SO}    \makeop{SL}    \makeop{SL}
\makeop{Cl}    \makeop{cond}  \makeop{Br}    \makeop{inv}   \makeop{rank}
\makeop{id}    \makeop{Fil}   \makeop{Frac}  \makeop{GL}    \makeop{SU}
\makeop{Trd}   \makeop{Sp}    \makeop{Tr}    \makeop{Trd}   \makeop{diag}
\makeop{Res}   \makeop{ind}   \makeop{depth} \makeop{Tr}    \makeop{st}
\makeop{Ad}    \makeop{Int}   \makeop{tr}    \makeop{Sym}   \makeop{can}
\makeop{length}\makeop{SO}    \makeop{torsion} \makeop{GSp} \makeop{Ker}
\makeop{Adm}
\def\makebb#1{\expandafter\def
  \csname bb#1\endcsname{{\mathbb{#1}}}\ignorespaces}
\def\makebf#1{\expandafter\def\csname bf#1\endcsname{{\bf
      #1}}\ignorespaces} 
\def\makegr#1{\expandafter\def
  \csname gr#1\endcsname{{\mathfrak{#1}}}\ignorespaces}
\def\makescr#1{\expandafter\def
  \csname scr#1\endcsname{{\EuScript{#1}}}\ignorespaces}
\def\makecal#1{\expandafter\def\csname cal#1\endcsname{{\mathcal
      #1}}\ignorespaces} 

\def\doLetters#1{#1A #1B #1C #1D #1E #1F #1G #1H #1I #1J #1K #1L #1M
                 #1N #1O #1P #1Q #1R #1S #1T #1U #1V #1W #1X #1Y #1Z}
\def\doletters#1{#1a #1b #1c #1d #1e #1f #1g #1h #1i #1j #1k #1l #1m
                 #1n #1o #1p #1q #1r #1s #1t #1u #1v #1w #1x #1y #1z}
\doLetters\makebb   \doLetters\makecal  \doLetters\makebf
\doLetters\makescr 
\doletters\makebf   \doLetters\makegr   \doletters\makegr
     \def\qed{\qedmark\medbreak}%
\def\qedmark{{\enspace\vrule height 6pt width 5pt depth 1.5pt}}%

\normalsize

\makeop{Bl}

\def\Spec{{\rm Spec}}

\def\Fpbar{\overline{\bbF}_p}
\def\Fp{{\bbF}_p}

\def\Zp{{\bbZ}_p}
\def\Qbar{\overline{\bbQ}}

\newcommand{\Z}{\mathbb Z}
\newcommand{\Q}{\mathbb Q}

\newcommand{\C}{\mathbb C}

\newcommand{\npr}{\noindent }

\newcommand{\<}{\langle}   
\renewcommand{\>}{\rangle} 

\nc{\embed}{\hookrightarrow}

\newcommand{\ch}{characteristic }
\newcommand{\ac}{algebraically closed }
\newcommand{\dieu}{Dieudonn\'{e} }

\nc{\ol}{\overline}
\nc{\wt}{\widetilde}
\nc{\opp}{\mathrm{opp}}
\def\ul{\underline}

\makeop{Ram}
\makeop{Rep}

\begin{document}
\renewcommand{\thefootnote}{\fnsymbol{footnote}}
\setcounter{footnote}{-1}
\numberwithin{equation}{section}


\title[Siegel threefold]
  {Geometry of the Siegel modular threefold with paramodular
  level structure}
\author{Chia-Fu Yu}
\address{
Institute of Mathematics \\
Academia Sinica \\
128 Academia Rd.~Sec.~2, Nankang\\ 
Taipei, Taiwan and NCTS (Taipei Office)}
\email{chiafu@math.sinica.edu.tw}

\date{\today}

\begin{abstract}
In this paper we extend some results of Norman and Oort and of de Jong, 
and give an explicit description of the geometry of the Siegel 
modular threefold with paramodular level structure. 
We also discuss advantages and restrictions of three standard methods
for studying moduli spaces of abelian varieties.
\end{abstract} 

\maketitle


\section{Introduction}
\label{sec:01}
In this paper we study the arithmetic model of the Siegel modular 
three-fold with
paramodular level structure at a prime $p$. This is a very special
case of geometric objects considered in the works of Norman and Oort
\cite{norman-oort}, and of de Jong \cite{dejong:ag}. We extend their
work and determine the singularities in this case. Our result can be
used for  computing the cohomology groups of the Siegel 3-fold by the
Picard-Leschetz formula (see \cite{sga7_2}, Expose XV).  \\

Let $p$ be a rational prime number, $N\ge 3$ a prime-to-$p$
positive integer. We choose a primitive $N$th root of unity $\zeta_N$ in
$\Qbar\subset \C$ and an embedding $\Qbar\embed \Qbar_p$.
Let $\bfA_{2,p,N}$ denote the moduli scheme over $\Z_{(p)}[\zeta_N]$ of
polarized abelian surfaces $(A,\lambda,\eta)$ with polarization degree
$\deg \lambda=p^2$ and with a symplectic level $N$-structure with
respect to $\zeta_N$. We denote by $\calA_{2,p,N}:=\bfA_{2,p,N}\otimes
\Fpbar$ the reduction modulo $p$ of the moduli scheme $\bfA_{2,p,N}$. \\

In \cite{norman-oort} Norman and Oort studies the {\it $p$-rank
  stratification} on Siegel moduli spaces (the case for {\it principally}
  polarized abelian varieties was studied earlier in Koblitz
  \cite{koblitz:thesis}) which we recall as follows. An abelian
  variety a field $K$ of \ch $p$ is said to have 
$p$-rank $f$ if $\dim_{\Fp} A[p](\ol K)=f$, which is an integer
between zero and $\dim A$. Conversely, given any integer $0\le f\le g$,
  there is a $g$-dimensional abelian variety with $p$-rank $f$. The
  discrete invariant $p$-rank defines locally closed subsets in the
  moduli space; it is showed in \cite{norman-oort} 
  that these subsets form a
  stratification. We now focus on the Siegel 3-fold. 
For each integer $0\le f\le 2$, let
$V_f\subset \calA_{2,p,N}$
(resp. $V_{\le f}\subset \calA_{2,p,N}$) be the reduced subscheme
consisting of points with $p$-rank equal to $f$ (resp. equal to or
less than $f$). The ordinary locus, which parametrizes the ordinary
  abelian varieties, is the $p$-rank 2 stratum; the
non-ordinary locus is $V_{\le 1}$. We have the following fundamental 
results for $\bfA_{2,p,N}$ (see \cite{oort:olso}, Theorem 2.3.3,
  \cite{norman-oort}, Theorems 3.1 and 4.1, \cite{dejong:ag},
  Theorem 1.12 and Proposition 3.3): 

\begin{thm}\label{11}\

{\rm (1)} The structure morphism 
$\bfA_{2,p,N}\to \Spec\Z_{(p)}[\zeta_N]$ is 
flat and a complete intersection morphism of relative
dimension 3. The fibers are geometrically irreducible. 

{\rm (2)} The stratum $V_f$ is dense in $V_{\le f}$ and $\dim V_f=1+f$
    for $f=0,1,2$. 

{\rm (3)} The non-ordinary locus $V_{\le 1}$ has two irreducible
components.   
  
\end{thm}

Let $\calS_{2,p,N}\subset \calA_{2,p,N}$
denote the supersingular locus, which is the reduced closed subscheme
over $\Fpbar$ consisting of supersingular abelian varieties. 
Recall that an abelian variety $A$ in \ch $p$ is called 
{\it supersingular} if it is isogenous to a product of 
supersingular elliptic curves over an \ac field $k$; 
it is called {\it superspecial} if it is isomorphic to a product 
of supersingular elliptic curves over $k$. Let $\Lambda_{2,1,N}\subset
\calA_{2,1,N}(\Fpbar)$ be the set of superspecial points in the moduli
space $\calA_{2,1,N}$ of principally polarized abelian surfaces with
level $N$-structure. Let $\Lambda\subset \calS_{2,p,N}(\Fpbar)$ be the
subset consisting of superspecial points $(A,\lambda,\eta)$ 
such that $\ker \lambda\simeq \alpha_p\times \alpha_p$. 
The following is a description of the supersingular
locus (see \cite{yu:ss_siegel}, Theorem 4.7):

\begin{thm}\label{12} \
   
{\rm (1)} The scheme $\calS_{2,p,N}$ is equi-dimensional and
  each irreducible component is isomorphic to
  $\bfP^1$ over $\Fpbar$.

{\rm (2)} The scheme $\calS_{2,p,N}$ has $|\Lambda_{2,1,N}|$
irreducible components. 

{\rm (3)} The singular locus of $\calS_{2,p,N}$ is the subset $\Lambda$.
Moreover, at each singular point there are $p^2+1$ irreducible
components passing through and any two of them intersect transversely.

\end{thm}

See \cite{yu:ss_siegel}, Corollary 3.3 and Theorem 4.1 for a precise
formula for $|\Lambda_{2,1,N}|$ and $|\Lambda|$, respectively. 

In this paper we prove 

\begin{thm}\label{13}
  The moduli scheme $\bfA_{2,p,N}$ is regular and smooth over $\Spec
  \Z_{(p)}[\zeta_N]$ exactly away from the subset $\Lambda$. At each singular
  point $x$ (in $\Lambda$) in the special fiber $\calA_{2,1,N}$, the
  formal completion of the local ring at $x$ is isomorphic to 
\[
  W(\Fpbar)[[t_{11},t_{12},t_{21},t_{22}]]/(p+t_{11}t_{22}-t_{12}t_{21}),
  \]
where $W(\Fpbar)$ is the ring of Witt vectors over
$\Fpbar$. Consequently, the moduli space $\calA_{2,p,N}$ is normal.
\end{thm}
 
Theorem~\ref{13} improves Theorem~\ref{11} (1). 
Theorems~\ref{11}-\ref{13} provide a good understanding of the
geometry of the moduli scheme $\bfA_{2,p.N}$. As the
singularities are non-degenerate ordinary double points, one can use
the Picard-Leschetz formula to compute the cohomology groups of the
moduli space $\bfA_{2,p,N}$. We intend to continue the work along this
direction. \\

The paper is organized as follows. In Section 2 we study the
singularity of the moduli scheme using the crystalline theory when
$p>2$. We also discuss the classification of deformations of
non-degenerate double points. In Section 3 we determine the singularity
of the moduli scheme using the local model for arbitrary residue
characteristic. In Section 4 we explain that some higher terms of the
defining equation in Norman's first example in \cite{norman:algo} is
required in order to describe the non-ordinary locus. 

\section{Singularities using the crystalline theory}
\label{sec:02}

In this section we use the crystalline theory to determine the
singularities of the moduli scheme $\bfA_{2,p.N}$ for $p>2$. Standard
references for the Grothendieck-Messing deformation theory are
\cite{grothendieck:bt} and \cite{messing:bt}.

\subsection{The smooth locus of $\bfA_{2,p,N}$}
\label{sec:21}
Let $k$ be an \ac field of \ch $p$. Let $W:=W(k)$ be the ring of Witt
vectors over $k$, and let $B(k)$ be the fraction field of $W(k)$. Let
$\sigma$ be the Frobenius map on $W$ and $B(k)$, respectively. 
For a $W$-module $M$ and a subset $S\subset M$, we denote by $<S>_W$
the $W$-submodule generated by $S$. Similarly, $<S>_{B(k)}\subset
M\otimes \Q_p$ denotes the vector subspace over $B(k)$ generated by
$S$. In this paper we use the covariant \dieu theory. Let $A$ be an
abelian variety over $k$. Denote by $M(A)$ the \dieu module of $A$,
which is the linear dual $\Hom_W(M^*(A), W(1))$ of the classical \dieu
module $M^*(A)$ of $A$. \dieu modules considered here are 
finite and free as $W$-modules.  

For convenience of discussion, we introduce the following definition.
\begin{defn}\label{21}
  A point $(A,\lambda,\eta)$ in $\bfA_{2,p,N}(k)$ is called {\it
  Lagrangian} if there exists a $W$-basis $X_1,X_2,
  Y_1,Y_2$ for $M=M(A)$ such that $Y_1,Y_2\in VM$,
\[ \<X_1,Y_1\>=-\<Y_1,X_1\>=1,\quad  \<X_2,Y_2\>=-\<Y_2,X_2\>=p. \]
and the other pairings are zero. 
\end{defn}

\begin{remark}\label{22}
  An equivalent definition (for general Siegel moduli spaces) is that
  a quasi-polarized \dieu module $M$ is called Lagrangian if there is
  a maximally isotropic $W$-sublattice $L\subset M$ such that $L$ is
  co-torsion free and contained in $VM$. This is also equivalent to
  the condition that the corresponding polarized abelian variety can
  be lifted over $W$ (cf. \cite{yu:lift}).  
\end{remark}

\begin{lemma}\label{23}
  If $x=(A,\lambda,\eta)\in\bfA_{2,p,N}(k)$ is a Lagrangian point,
  then the moduli space $\bfA_{2,p,N}$ is smooth at $x$.  
\end{lemma}
\begin{proof}
  We choose a $W$-basis $X_1,X_2, Y_1,Y_2$ for $M=M(A)$ as in
  Definition~\ref{21}. We compute the Hodge filtration ${\rm Fil}$ of
  $H_1^{\rm DR}(A)=M/pM$:
\[ {\rm Fil}=VM/pM=<Y_1,Y_2>_k. \] 
  The first order
  universal deformation of $M$ is 
  given by $\wt{\rm Fil}=<\wt Y_1,\wt Y_2>_R$, where 
\[ \wt Y_1=Y_1+t_{11}X_1+t_{12}X_2,\quad \wt
  Y_2=Y_2+t_{21}X_1+t_{22}X_2,\]
and $R$ is the first order universal deformation ring. 
One computes $\<\wt Y_1,\wt Y_2\>=-t_{21}+pt_{12}$. This shows
\[
R=W(k)[[t_{11},t_{12},t_{21},t_{22}]]/((-t_{21}+pt_{12})+\grm^2),\]
where $\grm$ is the maximal ideal of
$W(k)[[t_{11},t_{12},t_{21},t_{22}]]$, and hence 
the smoothness. \qed
\end{proof}

Recall that the {\it $a$-number} of a \dieu module $M$ over $k$,
denote $a(M)$, is defined to be $\dim_k M/(F,V)M$. 
Similarly, define $a(A):=a(M(A))$ to be the $a$-number of an abelian
variety $A$.  
 
Let $x=(A,\lambda,\eta)$ be a point in $\bfA_{2,p,N}(k)$ and let $M$
be the associated quasi-polarized \dieu module. We have the following
cases. \\

(i) {\bf $a(A)\le 1$.} It follows from \cite{norman-oort}, Theorem
2.3.3 that  $x$ is a Lagrangian point. \\

(ii) {\bf $a(A)=2$.} In this case $A$ is superspecial. It is not hard to
classify the \dieu module $M$. There are two possibilities 
(notice that $k$ is algebraically closed):\\

(iia) There is a $W$-basis $X_1,X_2,Y_1,Y_2$ for $M$ with the
following properties: 
\[ FX_1=Y_1,\ FY_1=-pX_1,\ FX_2=Y_2, \ FY_2=-pX_2, \]
\[ \<X_1,Y_1\>=-\<Y_1,X_1\>=1,\quad  \<X_2,Y_2\>=-\<Y_2,X_2\>=p, \]
and the remaining pairings are zero. \\

(iib) There is a $W$-basis $X_1,X_2,Y_1,Y_2$ for $M$ with the
following properties: 
\[ FX_1=Y_1,\ FY_1=pX_1,\ FX_2=Y_2,\ FY_2=pX_2, \]
\[ \<X_1,X_2\>=-\<X_2,X_1\>=1,\quad  \<Y_1,Y_2\>=-\<Y_2,Y_1\>=p, \]
and the remaining pairings are zero. 

\begin{lemma}\label{24}
  A point $x=(A,\lambda,\eta)$ of $\bfA_{2,p,N}(k)$ lies in the case
  {\rm (iib)} above if and only if it lies in the finite set $\Lambda$. 
\end{lemma}
\begin{proof}
  Recall that if $H$ is the kernel of an isogeny $\varphi:A\to B$ then
  the \dieu module $M(B)/M(A)$ is isomorphic to $M^*(H^D)$, 
  where $H^D$ is the Cartier
  dual of $H$. One only needs to consider the case (ii) as the
  $a$-number is equal to $2$
  for any member in $\Lambda$. 
  In the case (iia), one has 
\[ M^t=<X_1,Y_1,\frac{1}{p} X_2, \frac{1}{p} Y_2>. \]
We have $FM^t\not\subset M$ and $VM^t\not\subset M$. Therefore,
$\ker\lambda\not\simeq \alpha_p\times \alpha_p$. 

In the case (iib), one has 
\[ M^t=<X_1,X_2,\frac{1}{p} Y_1, \frac{1}{p} Y_2>. \]
We have $FM^t\subset M$ and $VM^t\subset M$. Therefore,
$\ker\lambda \simeq \alpha_p\times \alpha_p$. This proves 
the lemma. \qed
\end{proof}

Note that the point $x$ in the case (iia) is Lagrangian. We will show
(Lemma~\ref{29}) that the
moduli space has singularities at points in the case (iib). 
It follows 
from Lemma~\ref{22} that the point $x$ in the case (iib) is not
Lagrangian.

From the discussions above, we have proven
\begin{prop}\label{25}
  The smooth locus of the structure morphism $\bfA_{2,p,N}\to
  \Spec\Z_{(p)}[\zeta_N]$ is the complement of
  the finite subset $\Lambda$.
\end{prop}

\subsection{Classification of non-degenerate double points.}
\label{sec:22}
Recall that a quadratic form $Q(x)$ over a local ring $A$ is called
{\it non-degenerate} if the associated bilinear form
\[ B(x,y):=Q(x+y)-Q(x)-Q(y) \]
is a perfect pairing. We show

\begin{prop}\label{26}
  Let $A$ be a complete Noetherian local ring with separably closed 
  residue field $k$, and $\grm_A$ its maximal ideal. Let
  $R:=A[[x_1,\dots, x_n]]$ 
  be the ring of formal power series over $A$,
  $\grm_R$ its maximal ideal, and let $f(x)\in R$. Suppose 
  \begin{equation}
    \label{eq:21}
    f(x)\equiv a+Q(x) \mod \grm_R^3,
  \end{equation}
where $a$ is an element in $\grm_A$ and $Q(x)$ is a non-degenerate
quadratic form. Then there is an isomorphism 
\[ R/(f(x))\simeq R/(a'u+Q'(x)) \]
where $Q'(x)$ is any non-degenerate quadratic form over $A$, $a'$ is an
element in $\grm_A$ such that $a'\equiv a \mod \grm_A^3$, and $u$ is any unit
in $A$.  
\end{prop}

Since $A$ is strictly Henselian, any two non-degenerate quadratic
forms are equivalent. We may take $Q'(x)$ to be the standard split
form. By rescaling half of variables by suitable units
in $A$, one shows that
\[ R/(a'u+Q'(x))\simeq R/(a'+Q'(x)). \] 

\begin{lemma}\label{27}
  Let $R:=A[[x_1,\dots, x_n]]$ be the ring of formal power series over a
  local ring $A$. 
  Let $I$ be the ideal of $R$ generated by $x_1,\dots,x_n$ and $f(x)$
  be an element of $R$. Suppose that
  \begin{equation}
    \label{eq:22}
   f(x)\equiv a+Q(x) \mod I^3, 
  \end{equation}
where $Q(x)$ is a non-degenerate quadratic form and $a\in \grm_A$,
then there is an isomorphism 
\[ R/(f(x))\simeq R/(a+Q(x)). \]
\end{lemma}

\begin{proof}
  This is \cite{freitag-kiehl}, Chapter III, Proposition 2.3.
\end{proof}



\begin{lemma}\label{28}
  Let $A$ be a complete Noetherian local ring and $\grm_A$ be its
  maximal ideal. Let $R$, $\grm_R$, $I$, and $f(x)$ be as before. 
  Suppose that 
  \begin{equation}
    \label{eq:23}
    f(x)\equiv a+\sum_{i=1}^n a_ix_i+Q(x) \mod I^3, 
  \end{equation}
where $a, a_i$ are elements in $\grm_A$, and $Q(x)$ is a
non-degenerate quadratic form over $A$. Then there is an isomorphism
\begin{equation}
  \label{eq:24}
  R/(f(x))\simeq R/(a'+Q'(x)) 
\end{equation}
for some element $a'\in \grm_A$ and some non-degenerate quadratic form $Q'(x)$.
\end{lemma}
\begin{proof}
  This is \cite{freitag-kiehl}, Chapter III, Proposition 2.4. We
  sketch the proof as it will be used to prove
  Proposition~\ref{26}. We will find an element $b=(b_1,\dots,b_n)$
  with $b_i\in \grm_A$ such that the linear term of 
  $f(x+b)$ vanishes, and apply Lemma~\ref{27}.  We have
  \begin{equation}
    \label{eq:25}
    f(x+b)\equiv f(b)+\sum_{i=1}^n a_ix_i+B(x,b)+L_b(x)+Q'(x) \mod I^3,
  \end{equation}
where $L_b(x)$ is a linear form with coefficients in the ideal
$(b_1,\dots,b_n)^2$ and $Q'(x)$ is a quadratic form with $Q'(x)\equiv
Q(x) \mod \grm_A$. Since $B$ is non-degenerate, there exists an
element $b$ such that the linear term vanishes. This proves the lemma. \qed  
\end{proof}

\begin{remark}\label{285}
  (1) If $a_i\in \grm_A^r$ for all $i$, where $r$ is a positive integer,
  then each component $b_i\in \grm_A^r$ and hence $a'=f(b)\equiv a
  \mod \grm_A^{2r}$. 

  (2) The information $a'\equiv a \mod \grm^3_A$ in
      Proposition~\ref{26}, which is not provided in
      \cite{freitag-kiehl}, Chapter III, Proposition 2.4
      (Lemma~\ref{28}), is required in the proof of the main theorem. 
\end{remark}

  

\npr {\it Proof of Proposition~\ref{26}.} Condition (\ref{eq:21})
implies that 
\begin{equation}
  \label{eq:26}
  f(x)\equiv a'+ \sum_{i=1}^n a_ix_i+Q'(x) \mod I^3, 
\end{equation}
such that $a'\equiv a \mod \grm_A^3$, each $a_i\in \grm_A^2$ and
$Q'(x)\equiv Q(x) \mod \grm_A$. By Lemma~\ref{28}, we have, for an
appropriate element $b$,
\[ f(x+b)\equiv a''+Q''(x) \mod I^3, \]
where $a''$ is an element in $\grm_A$ such that $a''\equiv a' \mod \grm_A^4$
(since $a_i\in \grm_A^2$, see Remark~\ref{285}) and $Q''(x)\equiv Q'(x) \mod
\grm_A$. Therefore, by Lemma~\ref{27},
\[ R/(f(x))\simeq R/(a''+Q''(x)),  \]
where $Q''(x)$ is a non-degenerate quadratic form and $a''$ is an
element in $\grm_A$ such that $a''\equiv a \mod \grm_A^3$. \qed

\subsection{Singularities of $\bfA_{2,p,N}$.} \label{sec:23}
Let $x=(A,\lambda,\eta)$
  be a point in $\Lambda\subset \calA_{2,p,N}(\Fpbar)$. Let $M$ be the
  associated quasi-polarized \dieu module. By Lemma~\ref{24}, we
  choose a $W$-basis 
  $X_1,X_2,Y_1,Y_2$ for $M$ as in the case (iib). It is easy to
  compute that the Hodge filtration ${\rm Fil}\subset H_1^{DR}(A)$ is
  equal to $<Y_1,Y_2>_k$, where $k=\Fpbar$. 
  Let $R^{\rm u}$ be the universal deformation ring of the
  quasi-polarized \dieu module $M$. It follows from Theorem~\ref{11}
  (1) that 
\[ R^{\rm u}=R/(f), \quad
  R:=W[[t_{11},t_{12},t_{21},t_{22}]] \]  
for some power
  series $f$. By the Grothendieck-Messing deformation theory, the
  universal deformation of $M$ over $R^{\rm u}/\grm_{R^{\rm u}}^p$ 
is given by 
\[ \wt{\rm   Fil}=<\wt Y_1,\wt Y_2>_{R^{\rm u}/\grm_{R^{\rm u}}^p}\subset
  H_1^{crys}(A/({R^{\rm u}/\grm_{R^{\rm u}}^p}))  \]
which is isotropic with respect to $\<\, ,\>$, where  
\[ \wt Y_1=Y_1+t_{11}X_1+t_{12}X_2,\quad \wt Y_2=Y_2+t_{21}X_1+t_{22}X_2.\]
One computes the relation $\<\wt Y_1,\wt
Y_2\>=t_{11}t_{22}-t_{12}t_{21}+p=0$ in $R^{\rm u}$. 
This implies that 
\begin{equation}
  \label{eq:27}
  f\equiv p+(t_{11}t_{22}-t_{12}t_{21})\mod \grm_{R}^p
\end{equation}
This shows 
\begin{lemma}\label{29}
  If $x\in \Lambda$, then the moduli space $\calA_{2,p,N}$ is not
  smooth at $x$. 
\end{lemma}

\begin{thm}\label{210}
  Assume $p>2$. Let $x=(A,\lambda,\eta)$ be a point in $\Lambda$. The
  formal completion of the local ring of $\bfA_{2,p,N}$ at $x$ 
is isomorphic to 
\[  W(\Fpbar)[[t_{11},t_{12},t_{21},t_{22}]]/
(p+t_{11}t_{22}-t_{12}t_{21}), \] 
\end{thm}
\begin{proof}
  This follows immediately from Proposition~\ref{26} and (\ref{eq:27}).
\end{proof}

\section{Singularities using the local model method}
\label{sec:03}

In this section we use the method of local models to determine the
singularities of the moduli scheme $\bfA_{2,p.N}$, including the case
where $p=2$. Our references are de Jong \cite{dejong:gamma} and
Rapoport-Zink \cite{rapoport-zink}.

\subsection{Local model diagrams.} 
\label{sec:31} Let $\Lambda_1:=\Z_p^4$ and $e_1,\dots, e_{4}$ be the standard
basis. Let $\psi$ be the alternating pairing on $\Lambda_1$ so that
its representing matrix with respect to the standard basis is  
\[ 
\begin{pmatrix}
  0 & I' \\ - I' & 0
\end{pmatrix},\quad I'=
\begin{pmatrix}
  0 & p \\ 1 & 0
\end{pmatrix}. \]
Let $\bfM^{\rm loc}$ be the local model associated to the lattice $\Lambda_1$. 
This is the projective scheme over $\Zp$ 
representing the following functor. For any $\Zp$-scheme $S$, the set 
$\bfM^{\rm loc}(S)$ of its $S$-valued points is the set of locally
free $\calO_S$-submodules $\calF\subset \Lambda_{1}\otimes \calO_S$, 
locally on $S$ direct summands of $\Lambda_{1}\otimes \calO_S$, 
which are isotropic with respect to the pairing $\psi$. 

Let $\wt \bfA_{2,p,N}$ be the moduli space over $\Zp[\zeta_N]$ 
parametrizing equivalence classes of objects $(\ul A,\xi)$, where 
$\ul A$ is an object in $\bfA_{2,p,N}$ and
\[ \xi:H_1^{\rm DR}(A/S)\simeq \Lambda_1\otimes \calO_S \] 
is an isomorphism which preserves the polarizations. 
Write $\wt \calA_{2,p,N}$ for the reduction $\wt \bfA_{2,p,N}\otimes
\Fpbar$ modulo $p$. 

Let $\calG$ be the group scheme over $\Zp$
  representing 
  the functor $S\mapsto \Aut (\Lambda_1\otimes
  \calO_S,\psi)$. This group acts on the schemes $\wt \bfA_{2,p,N}$ and
  $\bfM^{\rm loc}$ from the left.
We have the following (local model) diagram 
\[ \xymatrix{
 & \wt \bfA_{2,p,N} \ar[ld]_{\varphi_1} \ar[rd]^{\varphi_2} & \\
\bfA_{2,p,N} & & \bfM^{\rm loc}\otimes \Zp[\zeta_N],   
} \]
where 
\begin{itemize}
\item $\varphi_2$ is the morphism that sends each object $(\ul
A_\bullet,\xi)$ to 
the image $\xi(\omega)$ of the Hodge submodule $\omega
\subset H_1^{\rm DR} (A/S)$, and
\item $\varphi_1$ is the morphism that forgets the trivialization $\xi$. 
\end{itemize}

We know that
\begin{itemize}
\item [(i)]the morphism $\varphi_2$ is $\calG$-equivalent, surjective and
  smooth, and has the same relative dimension as $\varphi_1$ does, and  

\item [(ii)] the morphism $\varphi_1:\wt \bfA_{2,p,N}\to \bfA_{2,p,N}$ 
  is a $\calG$-torsor.  
\end{itemize}

\subsection{Singularities of $\bfA_{2,p,N}$}
\label{sec:32}

Let $x=(A,\lambda,\eta)$
  be a point in $\Lambda\subset \calA_{2,p,N}(k)$, where $k=\Fpbar$. 
Let $M$ be the
  associated quasi-polarized \dieu module. By Lemma~\ref{24}, we
  choose a $W$-basis 
  $X_1,X_2,Y_1,Y_2$ for $M$ as in the case (iib). We choose 
  a trivialization $\xi$ which sends $Y_1,X_1, X_2, Y_2$ to $e_1,
  e_2,e_3, e_4$, respectively. This defines a point $y\in \wt
  \calA_{2,p,N}$ mapping to $x$. We put
  $z=\varphi_2(y)\in \bfM^{\rm loc}(\Fpbar)$, which is given by
\[ <e_1,e_4>_k\subset \Lambda_1\otimes k. \]
Around the point $z$ we choose coordinates same as in
Subsection~\ref{sec:23} so that an affine open chart $\calU$ is 
\[ \Spec
W[t_{11},t_{12},t_{21},t_{22}]/(p+t_{11}t_{22}-t_{12}t_{21}). \]

Let $R_x$ (resp. $R_y$ and $R_z$) be the completion of the local ring
of $\calA_{2,p,N}$ at $x$ (resp. of $\wt
  \calA_{2,p,N}$ at $y$ and of  $\bfM^{\rm loc}$ at $z$). We have 
  \begin{equation}
    \label{eq:31}
   R_x[[t_1,t_2,t_3]]\simeq R_y\simeq R_z[[s_1,s_2,s_3]], 
  \end{equation}
and 
\[ R_z\simeq
W[[t_{11},t_{12},t_{21},t_{22}]]/(p+t_{11}t_{22}-t_{12}t_{21}). \]
By \cite{dejong:gamma}, Lemma 4.7, there is an isomorphism
\begin{equation}
  \label{eq:32}
  R_x\simeq R_z\simeq
  W[[t_{11},t_{12},t_{21},t_{22}]]/(p+t_{11}t_{22}-t_{12}t_{21}).   
\end{equation}
This shows Theorem~\ref{13}. 

\begin{remark}
  The local model is a very effective method for computing local moduli
  spaces. However, one can apply it in the special case when the
  kernel of the polarization is contained in the $p$-torsion subgroup. 
\end{remark}


\section{Some fine information from the Cartier theory}
\label{sec:04}

As experts all know that there is another important tool for computing
local moduli spaces -- using the Cartier theory, we discuss our
situation with this method. First of all, the method works exclusively 
only in \ch $p$, namely the best information we can get is the 
local moduli spaces modulo $p$. In \cite{norman:algo} Norman
established the theoretic background for an algorithm for computing local 
moduli spaces. In principle the algorithm can only compute 
the coordinate ring of the universal deformation space modulo 
a specific power of the maximal ideal. 
However, since the singularity is determined by its
sufficiently well approximation, one would be able to determine 
the singularity eventually. 
Norman gave two examples in \cite{norman:algo} and the
first example is exactly the case of $\calA_{2,p,N}$ at points in
$\Lambda$.    

Let $M$ be the \dieu module associated to a point $x$ in $\Lambda$. 
We can choose basis $\{e_i\}$ for $M$ so that (cf. (iib)):
\begin{equation}
  \label{eq:41}
  \begin{split}
    Fe_1 =e_3,\quad  Fe_2=e_4, \\
    e_3=V e_1, \quad  e_4=Ve_2. \\  
  \end{split}
\end{equation}
and that the quasi-polarization $P$ from $M$ to its dual $M^t$ is
given by 
\[ P(e_1)=f_2, \quad P(e_2)=-f_1, \]
where $\{f_i\}$ is the dual basis for $M^t$. Using the theory of
displays, one constructs the equi-\ch universal deformation $\wt M$ over the
universal deformation ring $R_0=k[[t_{11}, t_{12},t_{21}, t_{22}]]$, which
is generated by $\{e_i\}$ over the Cartier ring ${\rm Cart}_p(R_0)$
(the ring $A_{R_0}$ in \cite{norman:algo}) 
with the following relations
\begin{equation}
  \label{eq:42}
  \begin{split}
    Fe_1&=e_3+T_{11}e_1+T_{12}e_2, \\
    Fe_2&=e_4+ T_{21}e_1+T_{22}e_2, \\
    e_3&=V e_1,\\ 
    e_4&=Ve_2, \\  
  \end{split}
\end{equation}
where $T_{ij}=[t_{ij}]$ is the Teichm\"uller lifting of $t_{ij}$.
Let $f\in R_0$ be the defining equation (up to a unit) 
of  the maximal closed formal subscheme over which the 
quasi-polarization $P$ extends. 
Norman showed that (when $p>2$):
\begin{equation}
  \label{eq:43}
  f= t_{11}t_{22}-t_{12}t_{21}+r(t),\quad \text{for some } r(t)\in
  (t_{ij})^p.  
\end{equation}
Since the singularity of $R_0/(f)$ is determined by its leading term, one
computes the local moduli space in \ch $p$. However, the following
lemma shows that the calculation of the remaining term $r(t)$
is required in order to determine the non-ordinary locus.
\begin{lemma}\label{41}
  The non-ordinary locus of the local moduli space $\Spec R_0/(f)$ is 
  $\Spec R_0/(t_{11}t_{22}-t_{12}t_{21},f)$.
\end{lemma}
\begin{proof}
  The Frobenius map on the tangent space $\wt M/V\wt M$ of the
  universal deformation $\wt M$ is given by
\[     Fe_1=t_{11}e_1+t_{12}e_2, \quad 
    Fe_2=t_{21}e_1+t_{22}e_2. \]
Therefore, the defining equation of the non-ordinary locus is
$t_{11}t_{22}-t_{12}t_{21}$. \qed
\end{proof}




\npr{\bf Acknowledgments.} The author thanks J.~Tilouine for his
interest on the work. The research was partially
supported by the grants NSC 97-2115-M-001-015-MY3 and AS-98-CDA-M01.

\end{document}